\documentclass{amsart}     
\usepackage{amssymb}

\newcommand{\QED}{$\Box$}
\newcommand{\Q}{\mathbb{Q}}

\newcommand{\R}{\mathbb{R}}
\newcommand{\N}{\mathbb{N}}

\newcommand{\dom}{\mbox{dom}}
\newcommand{\ran}{\mbox{ran}}
\newcommand{\diam}{\mbox{diam}}

\newtheorem{theorem}{Theorem}[section]
\newtheorem{lemma}[theorem]{Lemma}

\theoremstyle{definition}
\newtheorem{definition}[theorem]{Definition}

\theoremstyle{theorem}
\newtheorem{corollary}[theorem]{Corollary}

\theoremstyle{theorem}
\newtheorem{proposition}[theorem]{Proposition}

\theoremstyle{theorem}

\theoremstyle{theorem}

\theoremstyle{definition}

\theoremstyle{theorem}

\numberwithin{equation}{section}

\begin{document}

\title{Effective versions of local connectivity properties
}

\author{Dale Daniel}
\address{Department of Mathematics\\
Lamar University\\
Beaumont, Texas 77710}
\email{dale.daniel@lamar.edu}

\author{Timothy H. McNicholl}
\address{Department of Mathematics\\
Lamar University\\
Beaumont, Texas 77710}
\email{timothy.h.mcnicholl@gmail.com}
\begin{abstract}
We investigate, and prove equivalent, effective versions of local connectivity 
and uniformly local arcwise connectivity for connected and computably compact subspaces of Euclidean space.  We also prove that Euclidean continua that are computably compact and effectively locally connected are computably arcwise connected.
\end{abstract}


\maketitle

\section{Introduction}\label{sec:INTRO}

Computability theory is concerned with the theoretical and practical limitations of discrete computing devices as well as the effective content of mathematical theorems.  That is, when is the solution operator for a given class of problems amenable to computation by a discrete computing device or susceptible to an explicitly constructive description?  Such a theory requires precise mathematical foundations in order to achieve rigorous demonstration of its results.  For computation with discrete data, such as the natural or rational numbers, the foundations laid by the work of Turing, Church, and Kleene suffice.  The interested reader 
may find a historical survey of the genesis of these ideas in \cite{Feferman.2006}
 and detailed developments in standard texts such as \cite{Davis.1958}.  These notions are also sufficient for the exploration of the effective content of theorems in algebra.  See, \emph{e.g.} \cite{Ershov.Goncharov.Nerode.Remmel.1998.2}. 

When one wants to consider the theorems of analysis and topology, it is essential however to have a sound theory of computation with continuous data.  Such a theory should extend without overriding the framework for discrete data.  In addition, in it the fundamental mathematical notion of approximation should bridge the divide between the continuous and the discrete.  Several such theories are available.  For example, see \cite{Bishop.Bridges.1985}, \cite{Kalantari.Welch.1998}, \cite{Kalantari.Welch.1999}, \cite{Pour-El.Richards.1989}.  
We will base our work here on the Type Two Effectivity approach to computable analysis
as developed in \cite{Weihrauch.2000}.  However, many of our results could be translated into the framework of other approaches.

In order to investigate the effective content of a theory it is first necessary to formulate 
effective versions of its basic definitions.  Roughly speaking, an effective version of a property insists that we can actually compute from any entity for which the property holds all objects whose existence is thereby entailed.
Fundamental to topology and much of analysis are the notions of compact, closed, open, and connected set.  Effective versions of compact, closed, and open sets within the framework of Type Two Effectivity were explored in detail by V. Brattka and K. Weihrauch in \cite{Brattka.Weihrauch.1999}.  Local connectivity perhaps sits on a lower tier than these first four topological concepts, but nevertheless has played an important role in the development of topology and analysis.
A space is locally connected if each of its points has a local basis of connected open sets.
This property plays a crucial role in the characterization of space-filling curves; \emph{i.e.}, the Hahn-Mazurkiewicz Theorem \cite{Hahn.1914}, \cite{Mazurkiewicz.1920}. 
The notion of effective local connectivity first appeared in J. Miller's paper on 
effective embeddings of balls and spheres \cite{Miller.2002.4springer}.  More recently, it has been used
by V. Brattka in connection with computation of functions from their graphs \cite{Brattka.2008}.   

Here, we consider two effective versions of a cousin of local connectivity: uniform local arcwise connectivity.  Roughly speaking, a space is uniformly locally arcwise connected if all points sufficiently close in the space can be joined by an arc of arbitrarily small diameter.  Here, an \emph{arc} is a compact, connected set for which there are exactly two points with the property that the removal of either one of them from the set results in another connected set.  
There are at least two ways to create an effective version of this notion.  On the one hand, we may want to compute how close two points need to be in order to join them by
an arc of diameter below some given value.  On the other hand, we may also want to 
compute such an arc.  These simple observations lead to the notions of effective uniform local arcwise connectivity and strongly effective uniform local arcwise connectivity.
These are defined precisely in Section \ref{sec:TTE}.

Our main result is that on subsets of Euclidean space that are connected and computably compact, the notions of effective local connectivity, effective uniform local arcwise connectivity, 
and strongly effective uniform local arcwise connectivity are equivalent.  Roughly speaking, a subset of $\R^2$ is computably compact if it can be plotted with arbitrary precision by a 
discrete computing device.  A precise definition which also covers spaces of dimension greater than two is given in Section \ref{sec:TTE}.  See also \cite{Brattka.2008}.

One interpretation of this result is that effective local connectivity provides, for spaces 
which are computably compact, the precise amount of information necessary for the 
computation of arcs between points in the space.  
Another interpretation is that it provides an effective version of a classical result: every Peano continuum is uniformly arcwise connected (see, \emph{e.g.} \cite{Hocking.Young.1961}).  
By a Peano continuum is meant a space which is compact, metrizable, connected, and locally connected.

\section{Summary of pertinent notions and results from topology}

The material in this section is taken from Hocking and Young \cite{Hocking.Young.1961}.  

Let $d$ be the Euclidean metric on $\R^n$.  If $X \subseteq \R^n$ is bounded, then we let 
\[
\diam(X) = \sup\{d(x,y)\ |\ x,y \in X\}.
\]
If $X, Y \subseteq \R^n$ are closed, then we let
\[
d(X,Y) = \min\{d(x,y)\ |\ x \in X\ \wedge\ y \in Y\}.
\]
Let $B_\epsilon(p)$ denote the open ball of center $p$ and radius $\epsilon$. 
When $S \subseteq \R^n$, we let 
\[
B_\epsilon(S) = \bigcup_{p \in S} B_\epsilon(p).
\]

\begin{proposition}\label{prop:COMPACT}
If $C$ is a compact set that is contained in an open set $U \subseteq \R^n$, 
then $B_\epsilon(C) \subseteq U$ for some $\epsilon > 0$.
\end{proposition}

\begin{proof}
Suppose otherwise.  Then, for each positive integer $n$, there is a point 
$p_n \in B_{1/n}(C) - U$.  For each $n$, there is a point $q_n \in C$ such that 
$d(p_n, q_n) < 1/n$.  Let $q_{n_1}, q_{n_2}, \ldots$ be a convergent subsequence
of $q_1, q_2, \ldots$, and let $q \in C$ be its limit.  Then, 
$\lim_{k \rightarrow \infty} d(p_{n_k}, q_{n_k}) = 0$.  It follows that 
$\lim_{k \rightarrow \infty} p_{n_k} = q$.  Hence, $q \in \R^n - U$ since $U$ is open.  This is a contradiction.
\QED\end{proof}

An \emph{arc} is a 
homeomorphic image of $[0,1]$.  
We will call such a homeomorphism a \emph{parametrization} of the arc.  
If $A$ is an arc, then there are exactly two points in $A$ such that the removal of
either one of these points from $A$ yields a connected set.  
We call these points the \emph{endpoints of $A$}.  It follows that if 
$f$ is a parametrization of an arc $A$, then $f(0)$ and $f(1)$ are the
endpoints of $A$.  If $x,y$ are the endpoints of an arc $A$, then we say 
that $A$ is \emph{an arc from $x$ to $y$}.  

A topological space $X$ is \emph{arcwise connected} if for every distinct $x,y \in X$ 
there is an arc in $X$ from $x$ to $y$.  

Suppose $X$ is a topological space.  A subset $C$ of $X$ is \emph{connected} if there do not exist open sets $U,V$ such that $U \cap X$ and $V \cap X$ are non-empty and such that $U \cap V \cap X = \emptyset$.  A \emph{connected component} of $X$ is a connected subset of $X$ that is not a proper subset of any connected subset of $X$.  \emph{i.e.} a maximal connected subset of $X$.  

We now discuss local connectivity properties.

\begin{definition}\label{def:LOC.CONNECTED}
A topological space $X$ is \emph{locally connected (LC)} if 
for every $p \in X$ and every neighborhood of $p$, $U$, 
there is a connected neighborhood of $p$, $V$, such that 
$V \subseteq U$.  
\end{definition}

The following is Theorem 3.2 of \cite{Hocking.Young.1961}.

\begin{theorem}\label{thm:LC}
Let $X$ be a topological space.  Then, $X$ is locally connected if and only if for every open $U \subseteq X$, 
	each connected component of $U$ is open in $X$.
\end{theorem}

\begin{definition}\label{def:ULAC}
A metric space $(X,d)$ is \it uniformly locally arcwise connected (ULAC)\rm\ if for every $\epsilon > 0$ there exists $\delta > 0$ such that 
for all distinct $x, y \in X$ with $d(x,y) < \delta$, there exists an 
arc $A$ in $X$ from $x$ to $y$ whose diameter is less than 
$\epsilon$.
\end{definition}

A \emph{continuum} is a compact, connected, and metrizable topological space.
A \emph{Peano continuum} is a locally connected continuum.  By the Hahn-Mazurkiewicz Theorem, these are precisely the images of $[0,1]$ under 
continuous maps \cite{Hahn.1914}, \cite{Mazurkiewicz.1920}.   
The most pertinent results about Peano continua are the following.
Proofs can be found in Chapter 3 of \cite{Hocking.Young.1961}. 

\begin{theorem}\label{thm:PCARC}
Every connected open subset of a Peano continuum is arcwise connected.
\end{theorem}

By a simple application of the Lebesgue Number Theorem, one can now prove the following.

\begin{theorem}\label{thm:PCULAC}
Every Peano continuum is uniformly locally arcwise connected.
\end{theorem}

Theorem \ref{thm:PCARC} is proven by means of \emph{simple chains}.  These 
will be a valuable tool for us as well.  We define them here.  

\begin{definition}\label{def:SIMPLE.CHAIN}
Let $(U_1, \ldots, U_k)$ be a sequence of sets. 
\begin{enumerate}
	\item $(U_1, \ldots, U_k)$ is a \emph{chain} if 
$U_i \cap U_{i+1} \neq \emptyset$ whenever $1 \leq i < k$. 

	\item $(U_1, \ldots, U_k)$ is a \emph{simple chain} if 
 $U_i \cap U_j \neq \emptyset$ precisely when  
$|i - j| \leq 1$. 
\end{enumerate} 
\end{definition}

\begin{proposition}\label{prop:CHAIN.UNION}
If $(C_1, \ldots, C_k)$ is a chain of connected sets, then 
$C_1 \cup \ldots \cup C_k$ is connected.
\end{proposition}

\begin{proof}
Let $C = C_1 \cup \cdots \cup C_k$.  By way of contradiction, suppose $C$ is disconnected.  Then, there exist open sets $U,V$ such that 
$C \subseteq U \cup V$, $C \cap U \neq \emptyset$, $C \cap V \neq \emptyset$, and 
$C \cap U \cap V = \emptyset$.

We first observe that for each $1 \leq j \leq k$, $C_j \subseteq U$ or $C_j \subseteq V$.  For, let $1 \leq j \leq k$.  Hence, $C_j \subseteq U \cup V$ and $C_j \cap U \cap V = \emptyset$.
Since $C_j$ is connected, it follows that $C_j \subseteq U$ or $C_j \subseteq V$.

Without loss of generality, suppose $C_1 \subseteq U$.
Let $j_0$ be the largest integer such that $C_i \subseteq U$ whenever 
$1 \leq i \leq j_0$.  Hence, $j_0 < k$.  Let $p \in C_{j_0} \cap C_{j_0+1}$. 
Since $p \in C_{j_0}$, $p \in U$.  On the other hand, by maximality of 
$j_0$, $C_{j_0+1} \not \subseteq U$.  So, $C_{j_0+1} \subseteq V$.  Hence, 
$p \in C \cap U \cap V$- a contradiction.  It follows that $C$ is connected.
\QED\end{proof}

We define some associated terminology.

\begin{definition}\label{def:SCFROM}
Suppose $(U_1, \ldots, U_k)$ is a simple chain.
If 
\[
x \in U_1 - \bigcup_{1 < j \leq k} U_j, 
\]
and if 
\[
y \in U_k - \bigcup_{1 \leq j < k} U_j,
\]
then we say that $(U_1, \ldots, U_k)$ is a simple chain \emph{from $x$ to $y$}.
\end{definition}

\begin{definition}\label{def:SIMPLE.CHAIN.DIAM}
Suppose $(U_1, \ldots, U_k)$ is a simple chain of subsets of $\R^n$ and that each 
$U_j$ is bounded.
 The \emph{diameter} of $(U_1, \ldots, U_k)$ is 
the maximum of $\diam(U_1)$, $\ldots$, $\diam(U_k)$.
\end{definition}

A key fact about simple chains is the following which is Theorem 3.4 of \cite{Hocking.Young.1961}.

\begin{theorem}\label{thm:SIMPLE.CHAIN}
If $\{U_\alpha\}_{\alpha \in I}$ is a covering of a connected space $X$ by open
sets, and if $x,y \in X$, then there exist $\alpha_1, \ldots, \alpha_k \in I$ such that 
$(U_{\alpha_1}, \ldots, U_{\alpha_k})$ is a simple chain
from $x$ to $y$.
\end{theorem}

\section{Background from computable analysis and computable topology}\label{sec:TTE}

A \emph{rational interval} is an open interval whose endpoints are rational numbers.

A \emph{rational point} in $\R^n$ is a point whose co\"ordinates are all rational numbers.

 An \emph{$n$-dimensional rational box} is a set of 
the form $(a_1, b_1) \times \cdots \times (a_n, b_n)$ where 
$a_1, b_1, \ldots, a_n, b_n \in \Q$ and $a_i < b_i$ for all $i$.
So, when $n=1$, an $n$-dimensional rational box is a rational interval. 

A function $f : [0,1] \rightarrow \R^n$ is a \emph{rational polygonal curve} if
there are rational points $q_1, \ldots, q_k \in \R^n$ and rational numbers
\[
0 = t_0 < t_1 < \ldots < t_k = 1
\]
such that whenever $x \in [t_j, t_{j+1}]$, 
\[
f(x) = q_j + \frac{(x - t_j)}{(t_{j+1} - t_j)}(q_{j+1} - q_j).
\]

Let $\Sigma$ be a finite alphabet that contains $0,1$.   Let $f : \subseteq A \rightarrow B$ denote that $\dom(f) \subseteq A$ and $\ran(f) \subseteq B$.  
A \emph{a naming system} for a space $M$ is a surjection $\nu : \subseteq \Sigma^a \rightarrow M$ where
$a$ is either $*$ or $\omega$.  We will use the following naming systems only.  Precise definitions can be found in \cite{Weihrauch.2000}, but we will work with informal yet sufficiently rigorous summaries.
\begin{enumerate}
	\item $\rho^n$ for $\R^n$.  Informally, a $\rho^n$-name for a point $x \in \R^n$ is a 
	list of all rational boxes to which $x$ belongs.

	\item $\kappa_{mc}$ for the compact subsets of $\R^n$.  Informally, a 
	$\kappa_{mc}$-name of a compact $X \subseteq \R^n$ is a 
	list of all finite covers of $X$ by rational boxes \it each of which contains at least one point of $X$\rm.  As in  Section 5.2 of \cite{Weihrauch.2000}, such a covering will be called \emph{minimal}.
	
	\item $\delta_{C}$ for 
	$C([0,1], \R^n)$, the space of all continuous functions from $[0,1]$ into $\R^n$.  Informally, a $\delta_{C}$-name of a 
	continuous function $f :[0,1] \rightarrow \R^n$ is a sequence of 
rational polygonal curves $\{F_t\}_{t \in \N}$ which satisfies the condition
\[
s \geq t\ \Rightarrow\ d_{max}(F_t, F_s) \leq 2^{-t}
\]
and such that $f = \lim_{t \rightarrow \infty} F_t$.  Here, 
\[
d_{max}(f,g) = \max\{d(f(x), g(x))\ |\ x \in [0,1]\}.
\]

	\item A name of a function $f : \N \rightarrow \N$ is an oracle Turing machine which computes $f$.  Such a name consists of a code of a Turing machine $M$ and an oracle $A \subseteq \N$ such that $M$ computes $f$ when supplied with oracle $A$.

	\item We use any standard naming system for $\N$ such as the tally representation.  In this case, we identify each number with its name.
\end{enumerate}

Since these are the only naming systems we will use, we 
will suppress their mention when discussing the computability
of objects and functions. 

Suppose $X_0$ and $X_1$ are spaces for which we have established naming systems.
Informally, a function $f : \subseteq X_0 \rightarrow X_1$ is \emph{computable}
if there is a Turing machine $M$ with an input and output tape and with the property that whenever a name of a point $p \in \dom(f)$ is written on the input tape and $M$ is allowed to run indefinitely, $M$ writes a name of $f(p)$ on the output tape subject to the restriction that whenever $M$ writes a symbol on the output tape, it can not later change that symbol.  Details can be found in \cite{Weihrauch.2000}.

Accordingly, when $a,b \in \{*, \omega\}$ and $f : \subseteq \Sigma^a \rightarrow \Sigma^b$, a name of $f$ consists of an oracle Turing machine which computes $f$ in the sense above.  We use a special symbol which does not belong to $\Sigma$ to terminate finite strings.

We will make frequent use of the following principle: computation of maxima and 
minima of continuous functions on compact sets is computable. 
See, for example, Corollary 6.2.5 of \cite{Weihrauch.2000}.

We also make copious use of the Principle of Type Conversion.  Informally stated, this means that to compute a name of a function $f$ from data $X_1, \ldots, X_n$, it suffices to show that one can uniformly compute $f(x)$ from $X_1, \ldots, X_n, x$.  See \cite{Weihrauch.2000} for details.

If $X \subseteq \R^n$, then for each $z\in X$ and each $U \subseteq \R^n$ such that $z \in U \cap X$, 
let $C_z^X(U)$ denote the connected component of $z$ in $U \cap X$.  Before 
proceeding further, we make a small observation about these components.

\begin{proposition}\label{prop:CONTAINED}
If $U, V$ are subsets of $\R^n$ such that $z \in U \cap V \cap X$, and if $U \subseteq V$, then 
$C_z^X(U) \subseteq C_z^X(V)$.
\end{proposition}

\begin{proof}
Let $C = C_z^X(U)$.  Hence, $C$ is a connected subset of $U \cap X$.
Since $U \cap X$ is a subspace of $V \cap X$, $C$ is a connected set in 
$V \cap X$.  Since $z \in C$, it follows that $C \subseteq C_z^X(V)$.  
\QED\end{proof}

\begin{definition}\label{def:ELC}
Suppose $X \subseteq \R^n$.
\begin{enumerate}
	\item A \emph{local connectivity (LC) function for $X$} is a function $f : \N \rightarrow \N$ such that for every $k \in \N$ and every 
$p \in X$, 
\[
X \cap B_{2^{-f(k)}}(p) \subseteq C_p^X(B_{2^{-k}}(p)).
\]

	\item We say that $X$ is \emph{effectively locally connected} if 
it has a computable LC function.
\end{enumerate}
\end{definition}

Note that if $f$ is an LC function for $X$, and if $g \geq f$, then $g$ is an 
LC function for $X$.  Hence, from an LC function for $X$ we can compute 
an increasing LC function for $X$.

This definition of effective local connectivity is due to J. Miller \cite{Miller.2002.4springer}.  
A more general version of this definition, which looks much more like an effectivization of the usual
definition of local connectivity, is given in \cite{Brattka.2008}.
In fact, the definition given here looks much more like an effective rendition of the 
related concept of ``connected \emph{im kleinen}".  It is fairly well-known that these notions
are equivalent over compact metric spaces, and these are the only spaces we shall consider
in this paper.  A proof that Definition \ref{def:ELC} is equivalent over compact spaces to a straightforward 
effectivization of the usual definition of local connectivity over compact spaces will be given in a subsequent paper.

We propose two effective versions of the notion of a ULAC space.  The second is, \it prima facie\rm, a strengthening of the first.  We will later show they are equivalent.

\begin{definition}\label{def:EULAC}
Suppose $X \subseteq \R^n$.
\begin{enumerate}
	\item A \emph{ULAC function for $X$} is a function $f : \N \rightarrow \N$
with the property that for every $k \in \N$ and 
all distinct $x,y \in X$, if $d(x,y) \leq 2^{-f(k)}$, then there is an arc in $X$
from $x$ to $y$ whose diameter is smaller than $2^{-k}$.

	\item We say that $X$ is \emph{effectively uniformly locally arcwise connected (EULAC)} if it has a computable ULAC function.
\end{enumerate}
\end{definition}

Again, note that if $f$ is a ULAC function for $X$, and if $g \geq f$, then 
$g$ is a ULAC function for $X$.  Hence, from a ULAC function for $X$ we can compute 
an increasing ULAC function for $X$.

For the sake of stating the following definition, we note again that names can be encoded as points in $\Sigma^\omega$ where $\Sigma$ is a finite alphabet that contains $0,1$.  Details can be found in \cite{Weihrauch.2000}.  

\begin{definition}\label{def:SEULAC}
Suppose $X \subseteq \R^n$.
\begin{enumerate}
	\item A \emph{SULAC witness for $X$} is a pair of functions $(f,\Phi)$
such that $f$ is a ULAC function for $X$, $\Phi : \subseteq \N \times \Sigma^\omega \times \Sigma^\omega \rightarrow \Sigma^\omega$, and, for all $k \in \N$, whenever
$p,q$ are names of distinct points $x,y$ in $X$ such that $d(x,y) < 2^{-f(k)}$, 
$\Phi(k, p,q)$ is a name of a parametrization of an arc in $X$ from $x$ to $y$ whose diameter is smaller than $2^{-k}$.   

	\item $X$ is \it strongly effectively uniformly locally arcwise connected (SEULAC)\rm\ if it has a computable SULAC witness.
\end{enumerate}
\end{definition}

Our main result is that all three properties ELC, EULAC, and SEULAC are equivalent for computable Euclidean continua.

\section{ELC and EULAC spaces}

In this section, we show that $X$ is ELC if and only if it is 
EULAC.

Recall that a \emph{Lebesgue number} for a covering $\{U_\alpha\}_{\alpha \in I}$ of a
compact set $C$ is a number $\delta > 0$ such that whenever 
$x,y \in C$ and $d(x,y) < \delta$, there exists $\alpha \in I$ such that 
$x,y \in U_\alpha$.  The Lebesgue number Lemma asserts that every open cover of a compact metric space has a Lebesgue number.

\begin{lemma}[\bf Computable Lebesgue Number Lemma\rm]\label{lm:ELN}
It is possible to uniformly compute, from a name of a compact set $X \subseteq \R^n$ and a finite covering of $X$ by rational boxes, an $L \in \N$
such that $2^{-L}$ is a Lebesgue number for this covering.
\end{lemma}

\begin{proof}
Fix a name of $X$.
Suppose we are given a finite covering of $X$, 
\[
\{R_1, \ldots, R_m\}, 
\]
where $R_1, \ldots, R_m$ are rational boxes.

We first claim that for each $p \in X$, there is a rational box $S$ 
such that $p \in S$ and $B_\delta(S) \subseteq R_j$ for some $j$ where 
$\delta = \diam(S)$.  For, let $p \in R_j$.  Let $\delta'$ be a positive number 
such that $B_{\delta'}(p) \subseteq R_j$.  Let $S$ be a rational box of diameter less than $\delta'/2$ such that $p \in S \subseteq R_j$.
Let $\delta = \diam(S)$.  Suppose $q \in B_\delta(S)$.  Then, there exists
$q' \in S$ such that $d(q, q') < \delta < \delta'/2$.  But, 
$d(q', p) < \delta' / 2$.  So, $d(q, p) < \delta'$.  Hence, 
$q \in B_{\delta'}(p) \subseteq R_j$.  

So, we begin our computation by searching the name of $X$ until we find a finite covering of $X$ 
$\{S_1, \ldots S_k\}$ such that for each $i \in \{1, \ldots, k\}$ there exists $j \in \{1, \ldots, m\}$ such that 
\[
B_{\delta_i}(S_i) \subseteq R_j
\]
where $\delta_i = \diam(S_i)$.
It follows that this search must terminate.  Compute $L \in \N$ such that 
$2^{-L} < \delta_1, \ldots, \delta_k$.

We claim that $2^{-L}$ is a Lebesgue number for the covering 
$R_1, \ldots, R_m$.  For, suppose $p,q \in X$ are such that $d(p,q) < 2^{-L}$.
Then, if $p \in S_i$, and if $B_{\delta_i}(S_i) \subseteq R_j$, it follows
that $p$ and hence $q$ belong to $R_j$.  
\QED\end{proof}

\it Throughout the rest of this paper, we assume $X \subseteq \R^n$ is a continuum.\rm\ Henceforth, we abbreviate $C_z^X(U)$ with $C_z(U)$.

\begin{theorem}\label{thm:EULAC} 
From a name of $X$ and an LC function for $X$, we can uniformly compute a name of a ULAC function for $X$.
\end{theorem}

\begin{proof}
Let $f : \N \rightarrow \N$ be an LC function for $X$.
We can assume $f$ is increasing.

We compute a name of a function $g : \N \rightarrow \N$ as follows.  We use the Principle of Type Conversion.    
Let $k \in \N$ be given as additional input.  We search the name of $X$
until we find a minimal cover of $X$, $\{R_1, \ldots, R_m\}$,
with the additional property that the diameter of each $R_i$ is 
smaller than $2^{-f(k+1)}$.  By Lemma \ref{lm:ELN}, we can then uniformly 
compute $L \in \N$ such that $2^{-L}$ is a Lebesgue number for this 
covering.
Set $g(k) = L$.

We claim that $g$ is a ULAC function for $X$.  For, let $k \in \N$, and let 
$R_1, \ldots, R_m, L$ be as in the definition of $g(k)$ above.
Suppose $p,q$ are distinct points in $X$ such that $d(p,q) < 2^{-g(k)}$.
Hence, $p,q \in R_i$ for some $i$.  Hence, $d(p,q) < 2^{-f(k+1)}$.  
We must show that there is an 
arc in $X$ from $p$ to $q$ whose diameter is less than $2^{-k}$.  
Since $f$ is an LC function for $X$, we have
\[
q \in B_{2^{-f(k+1)}}(p) \cap X \subseteq C_p(B_{2^{-(k+1)}}(p)).
\]
It follows from Theorem \ref{thm:LC} that $C_p(B_{2^{-(k+1)}}(p))$ is an open
subset of $X$.  It then follows from Theorem \ref{thm:PCARC} that 
it is arcwise connected.  So, there is an arc from $p$ to $q$ in 
$C_p(B_{2^{-(k+1)}}(p))$.  Since this arc is contained in $B_{2^{-(k+1)}}(p)$, its diameter is less than $2^{-k}$.
\QED\end{proof}

\begin{corollary}
Every computably compact ELC continuum is EULAC.
\end{corollary}

We now prove the converse of Theorem \ref{thm:EULAC} holds for computable Euclidean continua.

\begin{theorem}\label{thm:EULACCONVERSE}
Every ULAC function for $X$ is an LC function for $X$.
\end{theorem}

\begin{proof}
Let $g$ be a ULAC function for $X$.
Let $n \in \N$, and let $y \in X$.  Let $C = C_y(B_{2^{-n}}(y))$.  
We claim that $X \cap B_{2^{-g(n)}}(y) \subseteq C$.
For, let $z \in B_{2^{-g(n)}}(y) \cap X$.  Then, there is an arc from 
$y$ to $z$ in $X$ of diameter less than $2^{-n}$.  
Hence, this arc is entirely contained in $B_{2^{-n}}(y)$. 
Thus, $z \in C$.
\QED\end{proof}
 
It now follows that if $X$ is a computably compact SEULAC continuum, then $X$ is ELC.

It only remains to show that every computably compact EULAC continuum is SEULAC.  To do so, we will
need the tools in the next section.

\section{Witnessing chains and arc chains}

Theorem \ref{thm:PCARC} is proven by building a descending sequence of simple chains of sets which are bounded, open, and connected.  In this section, we 
set forth a constructive version of the concept of a simple chain.  A first attempt
might be a sequence of names of open sets $U_1, \ldots, U_k$ such that 
$(U_1 \cap X, \ldots, U_k \cap X)$ is a simple chain of connected sets.
There are two difficulties with this definition.  The first is that we cannot determine from these data in finitely many steps whether $U_i \cap X$ is connected.
Nor, can we determine if $X \cap U_i \cap U_j = \emptyset$.  To overcome the first difficulty, we introduce the concept of a witnessing chain.

In this section and the next, we assume $f$ is an LC function for $X$.  We assume $f$ is increasing.  Hence, $2^{-f(k)} \leq 2^{-k}$ for all $k \in \N$.

\begin{definition}\label{def:WIT.CHAIN}
A \emph{witnessing chain} is a sequence $(m, R_1, \ldots, R_k)$ such that 
\begin{itemize}
	\item each $R_i$ is a rational box, 

	\item $R_i \cap R_{i+1} \cap X \neq \emptyset$ whenever 
$1 \leq i < k$, and
	
	\item $\diam(R_i) < 2^{-f(m)}$ whenever $1 \leq i \leq k$.
\end{itemize}
\end{definition} 

\begin{proposition}\label{prop:WIT.CHAIN.1}
Suppose $\omega = (m,R_1, \ldots, R_k)$ is a witnessing chain and that 
$1 \leq i \leq k$.  Then, for all $x,y \in R_i \cap X$, 
\[
C_x(B_{2^{-m}}(R_i)) = C_y(B_{2^{-m}}(R_i)).
\]
\end{proposition}

\begin{proof}
Let $U = B_{2^{-m}}(R_i)$.  
Since $\diam(R_i) < 2^{-f(m)}$, it follows that 
$y \in X \cap B_{2^{-f(m)}}(x)$.  Since 
$B_{2^{-m}}(x) \subseteq U$, it follows from Proposition \ref{prop:CONTAINED} that 
\[
C_x(B_{2^{-m}}(x)) \subseteq C_x(U).
\]
Since $f$ is an LC function for $X$, $X \cap B_{2^{-f(m)}}(x) \subseteq C_x(B_{2^{-m}}(x))$. It then follows that 
$y \in C_x(U)$.  So, $C_x(U)$ and $C_y(U)$ are connected components of $U \cap X$ having a point in common, namely $y$.  Hence, they are equal. 
\QED\end{proof}

We now make some notation.
Let $\omega = (m, R_1, \ldots, R_k)$ be a witnessing chain.  We let:
\begin{eqnarray*}
V_\omega & = & \bigcup_{i=1}^k B_{2^{-m}}(R_i)\\
m_\omega & = & m\\
k_\omega & = & k\\
R_{\omega, j} & = & R_j\\
\end{eqnarray*}

We also let 
\[
C_{\omega, i} = C_x(B_{2^{-m_\omega}}(R_{\omega, i}))
\]
for any $x \in R_{\omega, i} \cap X$.   By Proposition \ref{prop:WIT.CHAIN.1}, the choice of $x$ is inconsequential. 
We then let
\[
C_\omega = \bigcup_{j=1}^{k_\omega} C_{\omega,j}.
\] 

The key properties of the sets associated with a witnessing chain are summarized in the following proposition.

\begin{proposition}\label{prop:SETS.WIT.CHAIN}
Let $\omega$ be a witnessing chain.
\begin{enumerate}
	\item $(C_{\omega, 1}, \ldots, C_{\omega, k_\omega})$ is a 
chain of open, arcwise connected subsets of $X$.\label{prop:SETS.WIT.CHAIN.1}

	\item $C_\omega$ is an open and arcwise connected subset of $X$.\label{prop:SETS.WIT.CHAIN.2}

	\item $C_\omega \subseteq V_\omega$.\label{prop:SETS.WIT.CHAIN.3}
\end{enumerate}
\end{proposition}

\begin{proof}
That each $C_{\omega, j}$ is open follows from Theorem \ref{thm:LC}.
By definition, they are connected.  So, by Theorem \ref{thm:PCARC}, they are arcwise connected. 

By definition and by Proposition \ref{prop:WIT.CHAIN.1}, 
$R_{\omega, j}\cap X \subseteq C_{\omega, j}$.  Hence, by Definition
\ref{def:WIT.CHAIN}, $(C_{\omega, 1}, \ldots, C_{\omega, k})$
is a chain.  It then follows from Proposition \ref{prop:CHAIN.UNION} that 
$C_\omega$ is an open connected subset of $X$. Hence, again by Theorem \ref{thm:PCARC}, it is arcwise connected.

Part \ref{prop:SETS.WIT.CHAIN.3} follows directly from the definitions.
\QED\end{proof}

A witnessing chain can thus be viewed as an effective version of a chain of sets in that the 
fact that the required intersections are non-empty is explicitly witnessed.  It also provides an effective analog of the concept of an open, connected, and bounded subset of 
$X$ since, although the set $C_\omega$ is perhaps unknown, the set $V_\omega$ explicitly bounds this set.  A sufficient amount of effectivity is demonstrated by Proposition \ref{prop:CE}.  That a sufficiently large set of witnessing chains exists is demonstrated by the following lemma.

\begin{definition}\label{def:WIT.SAME}
Suppose $\omega = (m, R, R_1, \ldots, R_k)$ is a witnessing chain.  If
$x \in R_1 \cap X$, and if $y \in R_k \cap X$, then we say that 
$\omega$ is a \emph{witnessing chain from $x$ to $y$}.
\end{definition}

\begin{lemma}\label{lm:SAME.COMPONENT}\label{lm:GET.STARTED}
Let $U \subseteq \R^n$ be an open set, and suppose $x,y$ are distinct points in the same connected component of $U \cap X$.  Then, there is a witnessing chain
$\omega$ from $x$ to $y$ such that $\overline{V_\omega} \subseteq U$.
\end{lemma}

\begin{proof}
It follows from Theorems \ref{thm:LC} and \ref{thm:PCARC} that there is an arc $A$ from $x$ to $y$ which is contained in $U \cap X$.  By Proposition \ref{prop:COMPACT}, there exists 
$m$ such that $\overline{B_{2^{-m}}(A)} \subseteq U$.  For each $p \in A$, there is a rational box $S_p$ such that $p \in S_p$, $\overline{B_{2^{-m}}(S_p)} \subseteq U$, and $\diam(S_p) < 2^{-f(m)}$.  By Theorem \ref{thm:SIMPLE.CHAIN}, there exist $p_1, \ldots, p_t \in A$ such that $(A \cap S_{p_1}, \ldots, A \cap S_{p_t})$ is a simple chain from $x$ to $y$.  Set $\omega = (m, S_{p_1}, \ldots, S_{p_t})$. 
It follows that $\omega$ is a witnessing chain from $x$ to $y$.
\QED\end{proof}

\begin{definition}\label{def:ARC.CHAIN}
An \emph{arc chain} is a sequence of the form 
$(\omega_1, \ldots, \omega_l)$ where
\begin{itemize}
	\item $\omega_j$ is a witnessing chain,
	
	\item $(V_{\omega_1}, \ldots, V_{\omega_l})$ is a simple chain, and
	
	\item $X \cap R_{\omega_i, k_{\omega_i}} \cap R_{\omega_{i+1}, 1} \neq \emptyset$ whenever $1 \leq i < l$.
\end{itemize}
\end{definition}

We make some notation.  Let $\mathfrak{p} = (\omega_1, \ldots, \omega_l)$ be 
an arc chain.  We let:
\begin{eqnarray*}
l_{\mathfrak{p}} & = & l\\
\omega_{\mathfrak{p}, j} & =& \omega_j\\
V_{\mathfrak{p}, j} & = & V_{\omega_{\mathfrak{p}, j}}\\
C_{\mathfrak{p}, j} & = & C_{\omega_j}\\
C_{\mathfrak{p}} & = & \bigcup_j C_{\mathfrak{p}, j}
\end{eqnarray*}

\begin{proposition}\label{prop:SETS.ARC.CHAIN}
Suppose $\mathfrak{p}$ is an arc chain.
\begin{enumerate}
	\item $(C_{\mathfrak{p}, 1}, \ldots, C_{\mathfrak{p}, l_{\mathfrak{p}}})$
is a simple chain of open, arcwise connected subsets of $X$.

	\item $C_\mathfrak{p}$ is an open, arcwise connected 
subset of $X$.
\end{enumerate}
\end{proposition}

\begin{proof}
The third condition of Definition \ref{def:ARC.CHAIN} ensures that 
$(C_{\mathfrak{p}, 1}, \ldots, C_{\mathfrak{p}, l_{\mathfrak{p}}})$
is a chain.  By Proposition \ref{prop:SETS.WIT.CHAIN}, 
$C_{\mathfrak{p}, j} \subseteq V_{\mathfrak{p}, j}$ for each $j$.
So, the second condition of Definition \ref{def:ARC.CHAIN}
ensures that this is a simple chain.

It now follows from Theorem \ref{thm:LC} and Proposition \ref{prop:CHAIN.UNION} that 
$C_\mathfrak{p}$ is an open connected set.  Hence, 
by Theorem \ref{thm:PCARC}, $C_\mathfrak{p}$ is an
arcwise connected subset of $X$.
\QED\end{proof}

An arc chain can thus be viewed as a constructive version of a simple chain of open connected sets in that it contains explicit witnesses that the intersections which are required to be non-empty are non-empty, and, via the $V$-sets, witnesses that the intersections which are required to be empty are empty.

\begin{definition}\label{def:ACREFINES}\label{def:GOES.THROUGH}
We say that a simple chain $(U_1, \ldots, U_k)$ \emph{narrowly goes straight through}
a simple chain $(V_1, \ldots, V_l)$ if 
there are integers
\[
0 = t_0 < t_1 < t_2 < \ldots < t_{l-1} < t_l = k
\]
such that $\overline{U_j} \subseteq V_i$ whenever $t_{i-1} + 1 \leq j \leq t_i$.
If $\mathfrak{p}_0$ and $\mathfrak{p}_1$ are arc chains, then we say that $\mathfrak{p}_0$ \emph{narrowly goes straight through} $\mathfrak{p}_1$ if 
$(V_{\mathfrak{p}_0, 1}, \ldots, V_{\mathfrak{p}_0, l_{\mathfrak{p}_0}})$
narrowly goes straight through $(V_{\mathfrak{p}_1, 1}, \ldots, V_{\mathfrak{p}_1, l_{\mathfrak{p}_1}})$.
\end{definition}

\begin{proposition}\label{prop:CE}
It is possible to compute, from names of $f$ and $X$, enumerations of the following.
\begin{enumerate}
	\item The set of witnessing chains.\label{prop:CE.1}

	\item The set of arc chains.\label{prop:CE.2}

	\item The set of all pairs of arc chains $(\mathfrak{p}_0, \mathfrak{p}_1)$
such that $\mathfrak{p}_0$ narrowly goes straight through 
$\mathfrak{p}_1$.\label{prop:CE.3}
\end{enumerate}
\end{proposition}

\begin{proof}
Fix a name for $X$.  From this name we can computably extract a list of all rational boxes 
which intersect $X$.

We can then compute an enumeration of the set of all witnessing chains as follows.
List a tuple $(m, R_1, \ldots, R_k)$ whenever we find rational boxes
$S_1, \ldots,$\\ $ S_{k-1}$ such that 
\begin{itemize}
	\item $\emptyset \neq S_i \cap X$ whenever $1 \leq i < k$, 

	\item $S_i \subseteq R_i \cap R_{i+1}$ whenever $1 \leq i < k$, 
\end{itemize} 
and it is also the case that $\diam(R_i) < 2^{-f(m)}$ whenever
$1 \leq i \leq k$.

We can now compute an enumeration of the set of all arc chains as follows.  List a tuple
$(\omega_1, \ldots, \omega_l)$ whenever $\omega_1, \ldots, \omega_l$ have
all been listed as witnessing chains, $(V_{\omega_1}, \ldots, V_{\omega_l})$
is a simple chain, and we have found rational boxes $S_1, \ldots, S_{l-1}$
such that $S_i \cap X \neq \emptyset$ and $S_i \subseteq R_{\omega_i, k_{\omega_i}} \cap R_{\omega_{i+1}, 1}$ whenever $1 \leq i < l$.

The relation ``narrowly goes through" is computable.  So, it now follows that we can compute, uniformly in the given data, a enumeration of the set of all pairs of arc chains $(\mathfrak{p}_0, \mathfrak{p}_1)$
such that $\mathfrak{p}_0$ narrowly goes straight through 
$\mathfrak{p}_1$
\QED\end{proof}

\begin{definition}\label{def:DIAM}
If $\mathfrak{p}$ is an arc chain, then the \emph{diameter} of $\mathfrak{p}$ is 
defined to be the diameter of the simple chain 
$(V_{\mathfrak{p}, 1}, \ldots, V_{\mathfrak{p}, l_{\mathfrak{p}}})$.
\end{definition}

\begin{definition}\label{def:FROM}
If $\mathfrak{p}$ is an arc chain such that
$(V_{\mathfrak{p}, 1}, \ldots, V_{\mathfrak{p}, l_\mathfrak{p}})$ 
is a simple chain from $x$ to $y$, 
then we say that $\mathfrak{p}$ is an \emph{arc chain from $x$ to $y$}.
\end{definition}

\begin{theorem}\label{thm:REFINES}
Suppose $\mathfrak{p}$ is an arc chain from $x$ to $y$.  
Then, for every $\epsilon > 0$, there is an arc chain from $x$ to $y$ 
of diameter less than $\epsilon$ that narrowly goes straight through $\mathfrak{p}$.
\end{theorem}

\begin{proof}
Let $\omega_i = \omega_{\mathfrak{p}, i}$.  Let $l = l_\mathfrak{p}$.

We construct our refinement of $\mathfrak{p}$ as follows.  
Let $x'_j \in C_{\omega_j} \cap C_{\omega_{j+1}}$ when $1 \leq j < l$.  
Set $x_0' = x$ and $x_l' = y$.  Since $\mathfrak{p}$ is an arc chain from $x$ to $y$, it follows that 
$x_0' \neq x_1'$ and that $x_{l-1}' \neq x_l'$.  Since $(C_{\omega_1}, \ldots, C_{\omega_l})$ is a simple chain, it follows that $x_{j-1}' \neq x_j'$ whenever $1 < j < l$.  Hence, by Proposition \ref{prop:SETS.ARC.CHAIN}, for each $j \in \{1, \ldots, l\}$, there is an arc $B_j \subseteq C_{\omega_j}$ from $x_{j-1}'$ to $x_j'$.

When $1 \leq j < l$, let $x_j$ be the first point on $B_j$ that lies in $B_{j+1}$.  \emph{i.e.}, every point on $B_j$ between $x'_{j-1}$ and $x_j$ lies in $B_j - B_{j+1}$, and $x_j \in B_j \cap B_{j+1}$.  Let $x_0 = x_0'$, and let $x_l = x_l'$.  It follows that $x_{j-1} \neq x_j$ whenever $1 \leq j \leq l$.  So, let 
$A_j$ be the subarc of $B_j$ from $x_{j-1}$ to $x_j$ whenever $1 \leq j \leq l$.  It follows that $A_j \cap A_{j+1} = \{x_j \}$ and that $A_i \cap A_j = \emptyset$ whenever $|i -j | > 1$.

We now divide each $A_j$ into two or more subarcs $A_{j,1}, \ldots, A_{j, s_j}$ of diameter
less than $\epsilon/3$.  In addition, we label these arcs so that $A_{j,i} \cap A_{j, i'} \neq \emptyset$ precisely when $|i - i'| \leq 1$.  
Let $x_{j,0} = x_{j-1}$ and $x_{j, s_j} = x_j$.  Let $x_{j,i}$ denote the common 
endpoint of $A_{j,i}$ and $A_{j,i+1}$ when $1 \leq i < s_j$.  
Choose a positive number $\delta$ such that 
$B_\delta(A_{j,i}) \cap B_\delta(A_{j',i'}) = \emptyset$ whenever 
$A_{j,i} \cap A_{j',i'} = \emptyset$.  In addition, choose $\delta$ small enough so
that $x \not \in B_\delta(A_{j,i})$ whenever $(j,i) \neq (1,1)$ and $y \not \in B_\delta(A_{j,i})$ whenever $(j,i) \neq (l,s_l)$. In addition, choose $\delta$ small enough so that $\delta < \epsilon /3$ and so that $\overline{B_\delta(A_{j,i})} \subseteq V_{\omega_j}$ for all $j,i$.  The latter is possible by Proposition \ref{prop:COMPACT}.  It follows that 
$(B_\delta(A_{1,1}), \ldots, B_\delta(A_{l,s_l}))$ is a simple chain from $x$ to $y$ whose diameter is smaller than $\epsilon$.  
 
 We now apply Lemma \ref{lm:GET.STARTED} to $x_{j,i-1}$, $x_{j,i}$, and $B_\delta(A_{j,i})$ and obtain a witnessing chain $\omega_{j,i}$ from $x_{j,i-1}$ to $x_{j,i}$ such 
that 
\[
\overline{V_{\omega_{j,i}}} \subseteq B_\delta(A_{j,i}) \subseteq V_{\omega_j}.
\] 

It follows that $(V_{\omega_{1,1}}, \ldots, V_{\omega_{l,s_l}})$ is a simple 
chain whose diameter is smaller than $\epsilon$ and that 
narrowly goes straight through $(V_{\omega_1}, \ldots, V_{\omega_l})$.
Note that when $V_{\omega_{j_1, i_1}}$ is an immediate predecessor of 
 $V_{\omega_{j_2, i_2}}$ in this sequence, $x_{j_1, i_1}$ is in 
\[
R_{\omega_{j_1, i_1}, k_{\omega_{j_1, i_1}}} \cap R_{\omega_{j_2, i_2}, 1} \cap X.
\]
Hence, $(\omega_{1,1}, \ldots, \omega_{l,s_l})$ is an arc chain.

Since $x \not \in B_\delta(A_{j,i})$ when $(j,i) \neq (1,1)$, it follows that
\[
x \in V_{\omega_{1, 1}} - \bigcup_{(i,j) \neq (1,1)} V_{\omega_{i,j}}.
\]
It similarly follows that 
\[
y \in V_{\omega_l, s_l} - \bigcup_{(i,j) \neq (l, s_l)} V_{\omega_{i,j}}.
\]
Hence, $(\omega_{1,1}, \ldots, \omega_{l,s_l})$ is an arc chain from $x$ to $y$.
\QED\end{proof}

The proof of the following is a minor modification of the proof of Theorem 3.5 of \cite{Hocking.Young.1961}.

\begin{lemma}\label{lm:OPEN}
Suppose $U \subseteq \R^n$ is open and connected, and let $p,q$ be distinct rational points in $U$.  Then, there is a rational polygonal curve $h$ such that 
$h(0) = p$, $h(1) = q$, and $\ran(h) \subseteq U$.
\end{lemma}

\begin{theorem}\label{thm:PARAM1}
It is possible to uniformly compute, from names of $X$, $f$, distinct $x,y \in X$, and a witnessing chain $\omega$ from $x$ to $y$, a name of a parametrization of an arc $A \subseteq X \cap V_\omega$ from $x$ to $y$.
\end{theorem}

\begin{proof}
Set $\mathfrak{p}_0 = (\omega)$.

It now follows from Theorem \ref{thm:REFINES} and Proposition \ref{prop:CE} that from the given data, we can uniformly compute a sequence of additional arc chains $\frak{p}_1, \frak{p}_2, \ldots$ with the following properties.  
\begin{enumerate}
	\item $\frak{p}_{j+1}$ narrowly goes straight through $\frak{p}_j$.
	
	\item The diameter of $\frak{p}_{j+1}$ is less than  $2^{-j}$.
	
	\item There exists a rational box $R$ such that $x \in R$ and 
	$\overline{R} \subseteq V_{\mathfrak{p}_j,1}$.\label{pr:x.in}
	
	\item There exists a rational box $S$ such that $y \in S$ and 
	$\overline{S} \subseteq V_{\mathfrak{p}_j,l_{\mathfrak{p}_j}}$.\label{pr:y.in}	
\end{enumerate}
Let $V_{j,i} = V_{\mathfrak{p}_j, i}$ and 
$l_j = l_{\mathfrak{p}_j}$.
Since $\mathfrak{p}_{j+1}$ narrowly goes straight through $\mathfrak{p}_j$, there are integers 
\[
0 = s(j+1, 0) < s(j+1, 1) < \ldots < s(j + 1, l_{j} - 1) < s(j+1, l_{j}) = l_{j+1}
\]
such that $\overline{V_{j+1, i_1}} \subseteq V_{j, i_2}$ whenever 
$s(j+1, i_2 - 1) + 1\leq i_1 \leq s(j+1, i_2)$.  
Such integers can be computed uniformly from $i$ and the given data.  It follows from the proof of Theorem \ref{thm:REFINES} that we can choose these refinements so that $s(j+1, i-1) + 1 < s(j+1, i)$.  (Namely, since each $s_j$ is at least $2$.)

We now inductively define a family of intervals $\{I_{j,i}\}_{j \in \N, 1 \leq i \leq l_j}$.  
To begin, set $I_{0,1} = [0,1]$.  Fix $j \in \N$, and suppose that $I_{j,i}$ has been defined for all $i \in \{1, \ldots, l_j\}$.  Fix $i_2$ such that $1 \leq i_2 \leq l_j$.  Let $a, b$ denote the left and right endpoints of $I_{j, i_2}$ respectively.  Set $d = s(j+1, i_2) - s(j+1, i_2 - 1)$.  Whenever 
$s(j+1, i_2 - 1) + 1 \leq i_1 \leq s(j+1, i_2)$, we let 
\[
I_{j+1, i_1} = \left[a + \frac{i_1 - s(j+1, i_2 - 1) -1}{d} (b-a), a + \frac{i_1 - s(j+1, i_2 - 1)}{d}(b-a)\right].  
\]
For all $t \in [0,1]$ and $j \in \N$, let 
\begin{eqnarray*}
F_{t,j} & = & \{k\ :\ t \in I_{j,k}\}\\
S_{t,j} & = & \bigcup\{V_{j,k}\ :\ k \in F_{t,j}\}
\end{eqnarray*}
And let
\[
S_t = \bigcap_j S_{t,j}.
\]
For each $j$, $F_{t,j}$ consists either of a single number or a pair of consecutive integers.  Since $(V_{j+1,1}, \ldots, V_{j+1, l_j})$ is a simple chain, it follows that the diameter of 
$S_{t,j+1}$ is smaller than $2^{-j + 1}$.  Since $\mathfrak{p}_{j+1}$ narrowly goes straight through $\mathfrak{p}_j$, 
it follows that $\overline{S_{t,j+1}} \subseteq S_{t,j}$.  Thus, $S_t$ contains exactly one
point, and we define $h(t)$ to be this point.  

By definition, if $t \in I_{j,k}$, then $h(t) \in V_{j,k}$.  Since each $\mathfrak{p}_j$ is a simple chain, it follows from the construction of the $I_{j,k}$'s that $h$ is injective.
Since $x \in V_{j,1}$ for all $j$, $h(0) = x$.  Similarly, $h(1) = y$.

We now claim that $h$ can be computed from the given data.  For, whenever $I$ is a rational interval, let 
\[
S_{I,j} = \bigcup\{V_{j,k}\ :\ I \cap I_{j,k} \neq \emptyset\}.
\]
Given a name of a $t \in [0,1]$ as input, we list a rational rectangle $R$ as output whenever $I, j$ are found such that $t \in I$ and $R \supseteq S_{I,j}$.  Thus, $h(t)$ belongs to every rational rectangle listed.  Conversely, suppose $R$ is a rational rectangle that contains $h(t)$.  Then, there is a number $j$ such that $B_{2^{-j+1}}(h(t)) \subseteq R$.  Let $t \in I_{j,k_0}$, and let $I = I_{j, k_0}$.  Thus, $h(t) \in S_{I,j+1}$.  Also, the diameter of $S_{I,j+1}$ is smaller than
$2^{-j+1}$.  We conclude that $S_{I,j+1} \subseteq R$.  Hence, $R$ will eventually be listed.  Thus, a name of $h(t)$ is produced
by this process.
\QED\end{proof}

We are now in position to complete the proof of our main result.  

\section{From EULAC to SEULAC}
 
\begin{theorem}\label{thm:MAIN}
From a name of $X$ and a ULAC function for $X$, we can compute a SULAC witness
for $X$.
\end{theorem}

\begin{proof}
Compute $N_1 \in \N$ such that $n \times (2^{-N_1})^2 < 1$.  (Recall that $n$ is the dimension of the space we are working in.)
It follows that for all $k \in \N$, 
\[
n \times (2^{-(k + N_1)})^2 < 2^{-2k}.
\]
Set $N_0 = N_1 + 3$.  

Let $g$ be a ULAC function for $X$.  We assume $g$ is increasing.
Let $f(k) = g(k + N_0)$.
It follows that $f$ is an increasing ULAC function for $X$.  
It then follows from Theorem 
\ref{thm:EULACCONVERSE} that $f$ is an LC function for $X$.

Suppose we are given as input $k \in \N$ and names of distinct 
$x,y \in X$ such that $d(x,y) < 2^{-f(k)}$.
We first claim there exist rational boxes $R_0, S_0, S_1$ such that 
\begin{itemize}
	\item $x \in S_0$
	\item $y \in S_1$
	\item $\overline{S_0}, \overline{S_1} \subseteq R_0$
	\item $\diam(R_0) < 2^{-k}$
	\item $d(\overline{S_0}, \R^n - R_0) \geq 2^{-(k + N_0)}$
	\item $d(\overline{S_1}, \R^n - R) \geq 2^{-(k + N_0)}$
\end{itemize}
For, let $x = (x_1, \ldots, x_n)$, and let $y = (y_1, \ldots, y_n)$.
Hence, $|x_j - y_j| < 2^{-(k + N_0)} < 2^{-(k + N_1)}$.
Choose rational numbers $r_j, s_j$ so that $s_j < x_j, y_j < r_j$, 
\[
2^{-(k + N_0)} < \min\{x_j, y_j\} - s_j < 2^{-(k + N_0 -1)},
\]
and
\[
2^{-(k + N_0)} < r_j - \max\{x_j, y_j\} < 2^{-(k + N_0 -1)}.
\]
We then have that
\begin{eqnarray*}
r_j - s_j & = & (r_j - \max\{x_j, y_j\}) + (\max\{x_j, y_j\} - \min\{x_j, y_j\})\\
& &  + (\min\{x_j, y_j\} - s_j)\\ 
& < & 2^{-(k + N_0 - 1)} + 2^{-(k + N_0)} + 2^{-(k + N_0 - 1)}\\
& < & 3 \cdot 2^{-(k + N_0 -1)}\\
& < & 4 \cdot 2^{-(k + N_0 -1)}\\
& = & 2^{-(k + N_0 -3)}\\
& = & 2^{-(k + N_1)}.
\end{eqnarray*}
It follows that 
\[
\sum_{j=1}^n (r_j - s_j)^2 < 2^{-2k}.
\]
Set $R_0 = (s_1, r_1) \times \ldots \times (s_n, r_n)$.
It then follows that $\diam(R_0) < 2^{-k}$.  It also follows that 
$d(x, \R^n - R_0)$, $d(y, \R^n - R_0) > 2^{-(k + N_0)}$.  So, there exist
$S_0, S_1$ as required.

So, we begin by searching for $R_0$, $S_0$, $S_1$ with these properties.
Since $d(x,y) < 2^{-f(k)}$, there is an arc in $X$ from $x$ to $y$ of diameter
less than $2^{-(k + N_0)}$.  It follows that any such arc is contained in $R_0$.
Furthermore, the diameter of any arc contained in $R_0$ is smaller than
$2^{-k}$.

Now, $x,y$ are in the same connected component of $X \cap R_0$.  So, by Lemma \ref{lm:SAME.COMPONENT}, there is a witnessing chain $\omega$ from $x$ to $y$ such that $\overline{V_\omega} \subseteq R_0$.  Additionally, by Proposition \ref{prop:CE}, we can find such a chain through a search procedure.  Namely, we scan a computed list of witnessing chains while simultaneously scanning our names for $x,y$ until we find a witnessing chain $\omega$ and rational boxes $S_2,S_3$ such that $x \in S_2 \subseteq R_{\omega,1}$, $y \in S_3 \subseteq R_{\omega, k_\omega}$, and $\overline{V_\omega} \subseteq R_0$.  
We now apply Theorem \ref{thm:PARAM1}.
\QED\end{proof}

\begin{corollary}
Every computably compact EULAC continuum is SEULAC.
\end{corollary}

The above proof not only constructs an arc from $x$ to $y$, but in addition it constructs
a parametrization of that arc.  This leads to the following definition and corollaries.

\begin{definition}\label{def:AC}
\begin{enumerate}
	\item An \emph{arcwise connectivity (AC) operator} for $X$ is a function
$\Phi : \subseteq \Sigma^\omega \times \Sigma^\omega \rightarrow \Sigma^\omega$
such that whenever $p,q$ are names of distinct $x,y \in X$ respectively, 
$\Phi(p,q)$ is a name of a parametrization of an arc $A \subseteq X$ from $x$ to $y$.

	\item $X$ is \emph{computably arcwise connected} if it has a computable AC operator.
\end{enumerate}
\end{definition}

\begin{corollary}\label{cor:AC}
It is possible to uniformly compute, from names of $X$, an LC function for $X$, and an open $U \subseteq \R^n$ such that $U \cap X$ is connected, a name of an AC operator for $X \cap U$.
\end{corollary}

\begin{proof}
Given names of distinct $x,y \in U \cap X$, we first search for a witnessing chain 
$\omega$ and rational boxes $S_1, S_2$ such that $x \in S_1$, $y \in S_2$, $\overline{V_\omega} \subseteq U$, $\overline{S_1} \subseteq R_{\omega,1}$, and $\overline{S_2} \subseteq R_{\omega, k_\omega}$.   By Lemma \ref{lm:GET.STARTED}, this search must terminate.
We then apply Theorem \ref{thm:PARAM1}.
\QED\end{proof}

\begin{corollary}\label{cor:EAC}
Every computably compact continuum that is ELC is also computably arcwise connected.
\end{corollary}

\section{Computing parametrizations of arcs}

In \cite{Miller.2002.4springer}, J. Miller constructs an arc that is computable as a compact subset of 
$\R^2$ and has computable endpoints but has no computable parametrization.  Recently, 
Gu, Lutz, and Mayordomo, have strengthened this result by constructing an arc $A$
such that any computable function $f$ of $[0,1]$ onto $A$ retraces itself infinitely often \cite{GLM.2009.4springer}.  Thus, a name of an arc as a compact set, while it provides enough information to plot $A$ on a computer screen (see, \emph{e.g.}, Section 5.2 of \cite{Weihrauch.2000}), falls far short of what is sufficient for the computation of a parameterization of $A$.  We now show a local connectivity function provides the right amount of additional information.  Recall that a \emph{modulus of continuity} of a continuous $f : [0,1] \rightarrow \R^n$ is a function $g : \N \rightarrow \N$ such that $d(f(s), f(t)) < 2^{-k}$ whenever $k \in \N$ and $s,t \in [0,1]$ are such that $|s - t| \leq 2^{-g(k)}$. 

\begin{theorem}
From a name of a parametrization of an arc $A \subseteq \R^n$, we can compute 
a name of $A$ and an LC function for $A$.  Conversely, from a name of an arc
$A \subseteq \R^n$ and an LC function for $A$, we can compute a name of a 
parametrization of $A$.
\end{theorem}

\begin{proof}
Let $h$ be a parametrization of $A$.  From $h$, we can compute a name of $A$.  (See, \emph{e.g.}, Theorem 6.2.4.4 of \cite{Weihrauch.2000}.)
  It follows by say Theorem 6.2.7 of \cite{Weihrauch.2000} that we can compute, uniformly in the given data, a modulus of continuity for $h$, $m$. 
  Compute $h^{-1}$ and a 
modulus of continuity for $h^{-1}$, $m_1$.  Let $g = m_1 \circ m$.

We claim that $g$ is a local connectivity function for $A$.  
For, let $k \in \N$, and let $p_1 \in A$.  Let $x_1$ be the unique preimage of 
$p_1$ under $h$.  Let $I$ be the interval $(x_1 - 2^{-m(k)}, x_1 + 2^{-m(k)})$.
Let $C = f[I \cap [0,1]]$.  
We claim that 
\[
A \cap B_{2^{-g(k)}}(p_1) \subseteq C \subseteq B_{2^{-k}}(p_1).
\]
For, let $p_2 \in C$.  Let $x_2$ be the unique preimage of $p_2$ under $h$.
Then, $x_2 \in I$, and so $|x_1 - x_2| < 2^{-m(k)}$.  Hence, $d(p_2, p_1) < 2^{-k}$.
Thus, $C \subseteq B_{2^{-k}}(p_1)$.  

Now, suppose $p_2 \in A \cap B_{2^{-g(k)}}(p_1)$.  Again, let $x_2$ be the unique preimage of $p_2$ under $h$.  Then, by defnition of $g$, $|x_1 - x_2| < 2^{-m(k)}$.  
Hence, $x_2 \in I$.  Thus, $p_2 \in C$. 

Conversely, suppose we are given a name of $A$ and an LC function for $A$, $f$. 
Let $x,y$ be the endpoints of $A$.
Our first goal is to compute $x$ and $y$ from these data.

In what follows, we refer to the theorems and corollaries proved above, taking $A$ for the subset $X$ there.  We first prove a few preliminary claims.  Let $h$ be a homeomorphism of $[0,1]$ onto $A$.

By Theorem \ref{thm:EULAC}, we can uniformly compute from the given data a ULAC function for $A$, $g$. 

Suppose we have an arc chain for $A$, $\mathfrak{p}$, such that 
$A \subseteq \bigcup_j V_{\mathfrak{p}, j}$.  Let: 
\begin{eqnarray*}
\mathcal{I}_1(\mathfrak{p}, k) & = & \{i\ |\ d(\overline{V_{\mathfrak{p}, 1 } }, \overline{V_{\mathfrak{p}, i } }) < 2^{-k}\}\\
\mathcal{I}_2(\mathfrak{p}, k) & = & \{i\ |\ d(\overline{V_{\mathfrak{p}, l_\mathfrak{p}}}, 
\overline{V_{\mathfrak{p}, i}}) < 2^{-k}\}\\
T_j(\mathfrak{p}, k) & = & \bigcup_{i \in \mathcal{I}_j(\mathfrak{p}, k)} V_{\mathfrak{p}, i}
\end{eqnarray*}
Now, suppose the diameter of $\mathfrak{p}$ is less than $2^{-g(k)}$.  Then, 
\begin{eqnarray*}
\diam(T_j(\mathfrak{p}, k)) & < & 2^{-k+ 3}.\\
\end{eqnarray*}
We now claim that each endpoint of $A$ lies in at least one of 
$T_1(\mathfrak{p}, k)$, $T_2(\mathfrak{p}, k)$.  For, suppose by way of contradiction that $p_1$ is an endpoint of $A$ that does not
lie in either of these sets.  Choose $i$ such that $p_1 \in V_{\mathfrak{p}, i} \cap A$.
Let $p_2 \in V_{\mathfrak{p}, 1} \cap A$, and let 
$p_3 \in V_{\mathfrak{p}, l_\mathfrak{p}} \cap A$.  Let $A_1$ be the subarc of $A$ from $p_1$ to $p_2$.  Let $A_2$ be the subarc of $A$ from $p_1$ to $p_3$.  
Therefore, since $A$ is an arc, $A_1 \subseteq A_2$ or $A_2 \subseteq A_1$.  Without loss of generality, suppose $A_1 \subseteq A_2$.  
Let $A_{2,1}$ be the subarc of $A$ from $p_2$ to $p_3$ so that 
$A_2 = A_1 \cup A_{2,1}$ and $A_1 \cap A_{2,1} = \{p_2\}$.  We claim that 
$A_{2,1} \cap V_{\mathfrak{p}, i} \neq \emptyset$.   For, suppose otherwise.
Let 
\begin{eqnarray*}
U & = & A_{2,1} \cap \bigcup_{1 \leq j < i} V_{\mathfrak{p}, j}\\
V & = & A_{2,1} \cap \bigcup_{i < j \leq l_\mathfrak{p}} V_{\mathfrak{p}, j}
\end{eqnarray*}
Since $(V_{\mathfrak{p}, 1}, \ldots, V_{\mathfrak{p}, l_\mathfrak{p}})$ is a
simple chain, it follows that $U \cap V = \emptyset$.  On the other hand, since $A_{2,1} \subseteq A$, $U$ and $V$ are non-empty.  But, since $A_{2,1}$ is connected, this is a contradiction.
So, let $p_4 \in A_{2,1} \cap V_{\mathfrak{p}, i}$.  Therefore, 
$d(p_4, p_1) < 2^{-g(k)}$.  Let $B$ be the subarc of $A$ from $p_4$ to $p_2$.
The only arc on $A$ from $p_4$ to $p_1$ is $B \cup A_1$.  Since 
$d(\overline{V_{\mathfrak{p}, i}}, \overline{V_{\mathfrak{p}, 1}}) \geq 2^{-k}$, 
$d(p_4, p_2) \geq 2^{-k}$.  Therefore, $\diam(B \cup A_1) \geq 2^{-k}$ which 
contradicts the assumption that $g$ is a ULAC function for $A$.  So, $p_1 \in T_1(\mathfrak{p}, k) \cup T_2(\mathfrak{p}, k)$.

Now, suppose $k$ is such that $\sqrt{n}(2^{-k+2}) < d(x,y) / 2$.
Therefore, there exist rational boxes $R_1, R_2$ such that 
$T_j(\mathfrak{p}, k) \subseteq R_j$ and $\diam(R_j) \leq \sqrt{n}(2^{-k+2})$.  Hence, 
$R_1 \cap R_2 = \emptyset$.  For, if $p \in R_1 \cap R_2$, then 
$d(x,y) \leq d(x,p) + d(p,y) < d(x,y)$.

So, to compute the endpoints of $A$ we first search for an arc chain for $A$, $\mathfrak{p}$, and a $k \in \N$ as well as disjoint rational boxes $R_1, R_2$ such that $A \subseteq \bigcup_j V_{\mathfrak{p}, j}$, the diameter of $\mathfrak{p}$ is less than $2^{-g(k)}$, and 
$T_j(\mathfrak{p}, k) \subseteq R_j$.  

We can then compute a name for one endpoint of $A$ as follows.
Enumerate a rational box $R$ whenever we find a $k' \in \N$ and an arc chain for $A$, $\mathfrak{p}$, such that $A \subseteq \bigcup_j V_{\mathfrak{p}, j}$, the diameter of 
$\mathfrak{p}$ is less than $2^{-g(k')}$, and 
\[
T_j(\mathfrak{p}, k) \subseteq R \subseteq R_1.
\]
We can similarly compute a name of the other endpoint of $A$.  

We now apply Corollary \ref{cor:AC}.
\QED\end{proof}

The following is claimed but not proven in \cite{Miller.2002.4springer}.

\begin{corollary}
An arc $A \subseteq \R^n$ has a computable parametrization if and only if 
it is computable as a compact set and is effectively locally connected.
\end{corollary}

\section*{Acknowledgements}
The second author thanks his wife Susan for her support as well as Jack Lutz and the computer science department of Iowa State University.   We also thank the referees for their helpful comments.

\bibliographystyle{amsplain}      

\end{document}